\documentclass{article}

\usepackage{graphicx}        % standard LaTeX graphics tool
\usepackage{multicol}        % used for the two-column index
\usepackage[bottom]{footmisc}% places footnotes at page bottom

\usepackage{url}

\usepackage{amsthm}
\usepackage{amssymb}
\usepackage{amsmath}

\newtheorem{theorem}{Theorem}
\newtheorem{definition}[theorem]{Definition}
\newtheorem{lemma}[theorem]{Lemma}
\newtheorem{remark}[theorem]{Remark}
\newtheorem{proposition}[theorem]{Poposition}

\begin{document}

\title{Asymptotic analysis of maximum likelihood estimation of covariance parameters for Gaussian processes: an introduction with proofs}

% Use \titlerunning{Short Title} for an abbreviated version of
% your contribution title if the original one is too long
\author{Fran\c{c}ois Bachoc \\
Institut de math\'ematique, UMR5219; \\
 Universit\'e de Toulouse; \\
CNRS, UPS IMT, F-31062 Toulouse Cedex 9, France, \\ \url{francois.bachoc@math.univ-toulouse.fr}}

%
% Use the package "url.sty" to avoid
% problems with special characters
% used in your e-mail or web address
%
\maketitle

\abstract{This article provides an introduction to the asymptotic analysis of covariance parameter estimation for Gaussian processes. Maximum likelihood estimation is considered.
The aim of this introduction is to be accessible to a wide audience and to present some existing results and proof techniques from the literature.
The increasing-domain and fixed-domain asymptotic settings are considered. Under increasing-domain asymptotics, it is shown that in general all the components of the covariance parameter can be estimated consistently by maximum likelihood and that asymptotic normality holds. In contrast, under fixed-domain asymptotics, only some components of the covariance parameter, constituting the microergodic parameter, can be estimated consistently. Under fixed-domain asymptotics, the special case of the family of isotropic Mat\'ern  covariance functions is considered. It is shown that only a combination of the variance and spatial scale parameter is microergodic. A consistency and asymptotic normality proof is sketched for maximum likelihood estimators.}

\section{Introduction}
\label{sec:intro}

Kriging \cite{stein99interpolation,rasmussen06gaussian} consists of inferring the values of a (Gaussian) process given observations at a finite set of points.
It has become a popular method for a
large range of applications, such as geostatistics \cite{matheron70theorie}, numerical code approximation \cite{sacks89design,santner03design,bachoc16improvement}, calibration \cite{paulo12calibration,bachoc14calibration,kennedy2001bayesian}, global optimization \cite{jones98efficient}, and machine learning \cite{rasmussen06gaussian}.

If the mean and covariance function of the Gaussian process are known, then the unknown values of the Gaussian process can be predicted based on Gaussian conditioning \cite{rasmussen06gaussian,santner03design}.
Confidence intervals are associated to the predictions.
In addition, in the case where the observation points of the Gaussian process can be selected, efficient goal-oriented sequential sampling techniques are available, for instance for optimization \cite{jones98efficient} or estimation of failure domains \cite{bect12}.

Nevertheless, the mean and covariance functions are typically unknown, so that the above methods are typically carried out based on a mean and covariance function selected by the user, that differ from the true ones. Here we shall consider the case where the mean function is known to be equal to zero and the covariance function is known to belong to a parametric set of covariance functions. In this case, selecting a covariance function amounts to estimating the covariance parameter. Large estimation errors of the covariance parameter can be harmful to the quality of the above methods based on Gaussian processes. Hence, one may hope to obtain theoretical guarantees that estimators of the covariance parameters converge to the true ones.

Here we will review some of such guarantees in the case of maximum likelihood estimation \cite{rasmussen06gaussian,stein99interpolation}, that is the most standard estimation method of covariance parameters. The two main settings for these guarantees are the increasing and fixed-domain asymptotic frameworks. Under increasing-domain asymptotics, we will show that, generally speaking, the covariance parameter is fully estimable consistently and asymptotic normality holds. Under fixed-domain asymptotics only a subcomponent of the covariance parameter, called the microergodic parameter, can be estimated consistently. We will show that the microergodic parameter is estimated consistently by maximum likelihood in the case of the family of isotropic Mat\'ern covariance functions, with asymptotic normality. In both asymptotic settings, we will provide sketches of the proofs. We will also highlight the technical differences between the proofs in the two settings.

The rest of the article is organized as follows. Gaussian processes, estimation of covariance parameters and maximum likelihood are introduced in Section \ref{sec:notation}. Increasing-domain asymptotics is studied in Section \ref{section:increasing}.
Fixed-domain asymptotics is studied in Section \ref{sec:fixed:domain}. Concluding remarks and pointers to additional references are provided in Section \ref{sec:conclusion}. A supplementary material contains the asymptotic normality results for the Mat\'ern model and the expressions of means and covariances of quadratic forms of a Gaussian vector.

\section{Framework and notations}
\label{sec:notation}

\subsection{Gaussian processes and covariance functions}

We consider a Gaussian process $\xi : \mathbb{R}^d \to \mathbb{R}$. We recall that $\xi$ is a stochastic process such that for any $m \in \mathbb{N}$ and for any $u_1,\ldots,u_m \in \mathbb{R}^d$, the random vector $(\xi(u_1),\ldots,\xi(u_m))$ is a Gaussian vector \cite{rasmussen06gaussian}. Here and in the rest of the paper, $\mathbb{N}$ is the set of positive integers.

We assume throughout that $\xi$ has mean function zero, that is $\mathbb{E}( \xi( u ) ) = 0$ for $u \in \mathbb{R}^d$. Thus, the distribution of $\xi$ is characterized by its covariance function 
\[
(u,v) \in \mathbb{R}^{2d} \mapsto \mathrm{cov}(  \xi(u) , \xi(v)).
\]
We assume in all the paper that the covariance function of $\xi$ is stationary, that is there exists a function $k^\star : \mathbb{R}^d \to \mathbb{R}$ such that for $u,v \in \mathbb{R}^d$,
\[
 \mathrm{cov}(  \xi(u) , \xi(v)) = k^\star( u-v ).
\] 
In a slight abuse of language, we will also refer to $k^\star$ as the (stationary) covariance function of $\xi$. The function $k^\star$ is symmetric because for $u \in \mathbb{R}^d$, $k^\star(u) = \mathrm{cov}(\xi(u),\xi(0)) = \mathrm{cov}(\xi(0),\xi(u)) = k^\star(-u)$.
This function is positive definite in the sense of the following definition.

\begin{definition} \label{def:positive:definite}
A function $\phi : \mathbb{R}^d \to \mathbb{R}$ is positive definite if for any $m \in \mathbb{N}$ and for any $u_1,\ldots,u_m \in \mathbb{R}^d$, the $m \times m$ matrix $[\phi(u_i - u_j)]_{i,j=1,\ldots,m}$ is positive semi-definite.
\end{definition}

The function $k^\star$ is positive definite because the matrices $[k^\star(u_i - u_j)]_{i,j=1,\ldots,m}$ of the form of Definition \ref{def:positive:definite} are covariance matrices (of Gaussian vectors).

We then consider a set of functions $\{k_{\theta} ; \theta \in \Theta\}$ where $\Theta \subset \mathbb{R}^p$ and where for $\theta \in \Theta$, $k_{\theta}$ is a function from $\mathbb{R}^d \to \mathbb{R}$ that is symmetric and positive definite. We also call $k_{\theta}$ a covariance function and $\theta$ a covariance parameter for $\theta \in \Theta$.

The set $\{k_{\theta} ; \theta \in \Theta\}$ is a set of candidate covariance functions for $\xi$, that is, this set is known to the statistician who aims at selecting an appropriate parameter $\theta$ such that $k_{\theta}$ is as close as possible to $k^\star$.
In the rest of the paper, we will consider that $k^\star$ belongs to $\{k_{\theta} ; \theta \in \Theta\}$. Hence, there exists $\theta_0 \in \Theta$ such that $k^\star = k_{\theta_0}$. This setting is called the well-specified case in \cite{bachoc13parametric,Bachoc2013cross,bachoc2018asymptotic}. Under this setting, we have a classical parametric statistical estimation problem, where the goal is to estimate the true covariance parameter $\theta_0$.

\subsection{Classical families of covariance functions}

For $q \in \mathbb{N}$ and for a vector $x$ in $\mathbb{R}^q$, we let let $||x||$ be the Euclidean norm of $x$.
A first classical family of covariance functions is composed by the isotropic exponential ones with $\Theta \subset (0,\infty)^2$ and
\[
k_{\theta} (x) = \sigma^2 e^{ - \alpha || x ||},
\]
for $\theta =  (\sigma^2 , \alpha)$ and $x \in \mathbb{R}^d$. 
A second classical family is composed by the isotropic Gaussian covariance functions, with
$\Theta \subset (0,\infty)^2$ and
\[
k_{\theta} (x) = \sigma^2 e^{ - \alpha^2 || x ||^2},
\]
for $\theta =  (\sigma^2 , \alpha)$ and $x \in \mathbb{R}^d$. 

Finally, a third classical family is composed by the isotropic Mat\'ern covariance functions, with 
$\Theta \subset (0,\infty)^3$ and
\begin{equation} \label{eq:isotropic:matern}
k_{\theta} (x) = \frac{\sigma^2 2^{1-\nu}}{\Gamma(\nu)} \left (\alpha||x||
  \right )^{\nu} {\cal K}_{\nu} \left (\alpha ||x|| \right ),
\end{equation}
where $\Gamma$ is the gamma function, ${\cal K}_{\nu}$ is the modified Bessel function of the second kind,  for $\theta =  (\sigma^2 , \alpha , \nu)$ and $x \in \mathbb{R}^d$. These families of covariance functions, and other ones, can be found for instance in \cite{bevilacqua2019estimation,genton2015cross,gneiting2004stochastic,rasmussen06gaussian,santner03design,stein99interpolation}.
We remark that the isotropic exponential covariance functions are special cases of the isotropic Mat\'ern covariance functions with $\nu = 1/2$ \cite{stein99interpolation}.

For these three families of covariance functions, one can check that $k_{\theta}(0) = \sigma^2$ (in the Mat\'ern case the function is extended at zero by continuity). Hence $\sigma^2$ is called the variance parameter, because if $\xi$ has covariance function $k_{\theta}$ we have $\mathrm{var}(\xi(u)) = \sigma^2$ for $u \in \mathbb{R}^d$. In these three families of covariance functions, for $u,v \in \mathbb{R}^d$, if $\xi$ has covariance function $k_{\theta}$ we have that $\mathrm{cov}(\xi(u),\xi(v))$ depends on $\alpha ||u-v||$. Hence $\alpha$ is called the spatial scale parameter because changing $\alpha$ can be interpreted as changing the spatial scale when measuring differences between input locations of $\xi$. In the three examples, $k_{\theta}(x)$ is a decreasing function of $||x||$, thus a large $\alpha$ makes the covariance decrease more quickly with $||x||$ and provides a small spatial scale of variation of $\xi$. Conversely, a small $\alpha$ makes the covariance decrease more slowly and provides a large spatial scale of variation of $\xi$. 

Finally, for the family of Mat\'ern covariance functions, $\nu$ is called the smoothness parameter. To interpret this, for $\theta \in \Theta$, let us call spectral density the function $\hat{k}_{\theta} : \mathbb{R}^d \to \mathbb{R}$ such that for $u \in \mathbb{R}^d$
\[
k_\theta(u) = \int_{\mathbb{R}^d} 
\hat{k}_{\theta} (\omega) e^{\mathrm{i} \omega^\top u} d \omega,
\]
with $\mathrm{i}^2 = -1$. Under mild regularity assumptions, that hold for the three families above, the function $\hat{k}_{\theta}$ is the Fourier transform of $k_{\theta}$. When $k_{\theta}$ is a Mat\'ern covariance function, we have
\begin{equation} \label{eq:matern:FT}
\hat{k}_{\theta}(\omega)
=
\sigma^2
\frac{\Gamma(\nu+d/2) \alpha^{2 \nu}}{\Gamma(\nu) \pi^{d/2} }
%\frac{\sigma^2 (1/\alpha)^{2 \nu}}{((\alpha)^2+ \omega^2)^{\nu+d/2}}
\frac{ 1}{(\alpha^2+ ||\omega||^2 )^{\nu+d/2}},
\end{equation}
for $\omega \in \mathbb{R}^d$ \cite{gneiting2010matern}. Hence, we see that for larger $\nu$, the Fourier transform $\hat{k}_{\theta}(\omega)$ converges to zero faster as $||\omega|| \to \infty$, which implies that the function $k_{\theta}$ is smoother at zero (this function is already infinitely differentiable on $\mathbb{R}^d \backslash \{0\}$). This is why $\nu$ is called the smoothness parameter. 

There is an important body of literature on the interplay between the smoothness of the covariance function of $\xi$ and the smoothness of $\xi$ \cite{Adler81,adler1990introduction,Azais2009level}. In our case, if $\xi$ has an exponential covariance function, then it is continuous and not differentiable (almost surely and in quadratic mean).  If $\xi$ has a Gaussian covariance function, then it is infinitely differentiable (almost surely and in quadratic mean). The Mat\'ern covariance functions provide, so to speak, a continuum of smoothness in between these two cases. Indeed, consider $\xi$ with Mat\'ern covariance function with smoothness parameter $\nu >0$. Then $\xi$ is $m$ times differentiable (almost surely and in quadratic mean) if $\nu > m$.

\subsection{Maximum likelihood}

Consider a sequence $(s_i)_{i \in \mathbb{N}}$ of spatial locations at which we observe $\xi$, with $s_i \in \mathbb{R}^d$. Assume from now on that the locations $(s_i)_{i \in \mathbb{N}}$ are two-by-two distinct. Then, for $n \in \mathbb{N}$, we consider the Gaussian observation vector $y = (y_1,\ldots,y_n)^\top =(\xi(s_1) , \ldots , \xi(s_n))^\top$. 

We consider a family of covariance functions $\{k_{\theta} ; \theta \in \Theta \}$ and assume further that for $n \in \mathbb{N}$, the covariance matrix $R_{\theta} := [k_{\theta}(s_i-s_j)]_{i,j=1,\ldots,n}$ is invertible. Then, when $\xi$ has covariance function $k_{\theta}$, the Gaussian density of $y$ is
\[
\mathcal{L}_n( \theta )
=
\frac{1}{ \sqrt{ |R_{\theta}| } (2 \pi)^{n/2} }
e^{ -\frac{1}{2} y^{\top} R_{\theta}^{-1} y },
\]
with $|R_{\theta}|$ the determinant of $R_{\theta}$. The focus of this paper will be on maximum likelihood estimation. A maximum likelihood estimator is a (measurable) estimator of $\theta_0$ that satisfies
\begin{equation} \label{eq:ML}
\hat{\theta}_{\text{ML}}
\in 
\underset{\theta \in \Theta}{\mathrm{argmax}}
\mathcal{L}_n( \theta ).
\end{equation}
We remark that, in general, there may not be a unique estimator $\hat{\theta}_{\text{ML}}$ satisfying \eqref{eq:ML}. Furthermore, the existence of measurable estimators satisfying \eqref{eq:ML} is not a trivial problem. We refer for instance to \cite{gine2016mathematical,molchanov2005theory} on this point.

In this paper, we assume that there exists at least one measurable estimator satisfying \eqref{eq:ML} and the results hold for any choice of such an estimator. A notable particular case is when $\Theta = (0,\infty)$, $\theta = \sigma^2$ and $k_{\theta} = \sigma^2 k^\star$. In this case, there is a unique estimator satisfying \eqref{eq:ML} (see also \eqref{eq:hat:sigma:matern} in Section \ref{sec:fixed:domain}). In this special case, we can call $\hat{\theta}_{\text{ML}}$ the maximum likelihood estimator. In general, one may rather call it a maximum likelihood estimator.

It is convenient to consider the following decreasing transformation of the logarithm of the likelihood, 
\begin{equation} \label{eq:Ln}
L_n(\theta) = \frac{1}{n} \log (|R_{\theta}|) + \frac{1}{n}
y^\top R_{\theta}^{-1} y,
\end{equation}
for $\theta \in \Theta$. We have 
\[
\hat{\theta}_{\text{ML}}
\in 
\underset{\theta \in \Theta}{\mathrm{argmin}}
L_n( \theta ).
\]

The problem of studying the asymptotic properties of $\hat{\theta}_{\text{ML}}$ as $n \to \infty$ presents several differences compared to the most standard parametric estimation setting where the observations are independent and identically distributed \cite{van2000asymptotic}. Indeed, in our case the components of the observation vector $y$ are dependent, so the logarithm of the likelihood is not a sum of independent random variables. Furthermore, the likelihood function involves the quantities $|R_{\theta}|$ and $ R_{\theta}^{-1}$ for which, often, no explicit expressions exist. Finally, for asymptotic statistics with independent and identically distributed data, there is a single asymptotic setting as $n \to \infty$. Here there exist several possible asymptotic settings, depending on how the  spatial locations $s_1,\ldots,s_n$ behave as $n \to \infty$. The proof techniques and the results obtained strongly depend on the asymptotic setting. We will now review some results under the two main existing asymptotic frameworks: increasing-domain and fixed-domain asymptotics.

\section{Increasing-domain asymptotics} \label{section:increasing}

In Section \ref{section:increasing}, we assume that there exists a fixed $\Delta >0$ such that
\begin{equation} \label{eq:smallest:distance:Delta}
\inf_{\substack{i,j \in \mathbb{N} \\ i \neq j}}
||s_i - s_j|| \geq \Delta.
\end{equation}
This assumption is the main assumption considered in the literature for increasing-domain asymptotics (see \cite{bachoc14asymptotic} for instance and see also \cite{bachoc2018asymptotic} for one of the few exceptions). This assumption implies that the spatial locations $(s_i)_{i \in \mathbb{N}}$ are not restricted to a bounded set. The results and proofs that will be presented in Section \ref{section:increasing} can mainly be found in  \cite{bachoc14asymptotic}.

\subsection{Consistency} \label{sub:consistency}

Here the aim is to show that $\hat{\theta}_{\text{ML}}$ converges to $\theta_0$, weakly. 
We consider a general family of covariance functions $\{ k_{\theta} ; \theta \in \Theta \}$, where $\Theta$ is compact, that satisfies 
\begin{equation} \label{eq:sup:k:theta}
\sup_{\theta \in \Theta}
| k_{\theta}(x) | \leq 
\frac{C_{\sup}}{1+||x||^{d + C_{\inf}}}
\end{equation}
and
\begin{equation} \label{eq:sup:partial:k:theta}
\max_{s=1,2,3}
~
\max_{\substack{ i_1,\ldots , i_s = \\  1,\ldots,p}}
~
\sup_{\theta \in \Theta}
\left| \frac{\partial^s}{\partial \theta_{i_1},\ldots,\partial \theta_{i_s}} k_{\theta}(x) \right| \leq 
\frac{C_{\sup}}{1+||x||^{d + C_{\inf}}},
\end{equation}
where $0 < C_{\inf}$ 
and $C_{\sup} < \infty$
are fixed constants and for $x \in \mathbb{R}^d$.

We also assume that 
\begin{equation} \label{eq:condition:FT}
(\theta,\omega) \in \Theta \times \mathbb{R}^d \mapsto \hat{k}_{\theta}(\omega) \quad \text{is continuous and strictly positive.}
\end{equation}

The families of isotropic exponential, Gaussian and Mat\'ern covariance functions do satisfy \eqref{eq:sup:k:theta} and \eqref{eq:sup:partial:k:theta}, when $\Theta$ is compact, and $\nu$ is fixed for Mat\'ern. Indeed, these functions and their partial derivatives, with respect to $\sigma^2$ and $\alpha$, are exponentially decaying as $||x|| \to \infty$, where $x$ is their input. For the exponential and Gaussian covariance functions this can be seen simply and for the Mat\'ern covariance function, this follows from the properties of the modified Bessel functions of the second kind \cite{AS64}. Also, when $\Theta$ is compact, exponentially decaying functions bounding the covariance functions and their partial derivatives can be chosen uniformly over $\theta \in \Theta$ (see again \cite{AS64} for the Mat\'ern covariance functions).

These three families of covariance functions also satisfy \eqref{eq:condition:FT}. The expressions of the Fourier transforms of these covariance functions can be found for instance in \cite{gneiting2010matern} and \cite{stein99interpolation}.

Then the next lemma enables to control the term $R_{\theta}^{-1}$ in \eqref{eq:Ln}. We let $\lambda_{\inf}(M)$ be the smallest eigenvalue of a symmetric matrix $M$.

\begin{lemma}[Proposition D.4 in \cite{bachoc14asymptotic}, Theorem 5 in \cite{bachoc2016smallest}] \label{lem:smallest:eigenvalue}
Assume that \eqref{eq:smallest:distance:Delta}, \eqref{eq:sup:k:theta} and \eqref{eq:condition:FT} hold.
We have
\[
\inf_{n \in \mathbb{N}}
\inf_{\theta \in \Theta}
\lambda_{\inf} (R_{\theta}) >0.
\]
\end{lemma}
\begin{proof}[sketch]
We have, for $n \in \mathbb{N}$ and $\lambda_1 , \ldots , \lambda_n \in \mathbb{R}$,
\begin{align} \label{eq:quadratic:form:FT}
\sum_{i,j=1}^n
\lambda_i \lambda_j (R_{\theta})_{i,j}
& = 
\sum_{i,j=1}^n
\lambda_i \lambda_j 
k_{\theta}(s_i - s_j)
\notag
\\
& = 
\sum_{i,j=1}^n
\lambda_i \lambda_j
\int_{\mathbb{R}^d}
\hat{k}_{\theta}(\omega)
e^{\mathrm{i} \omega^\top (s_i - s_j)}
d \omega
\notag
\\
& = 
\int_{\mathbb{R}^d}
\hat{k}_{\theta}(\omega)
\left(
\sum_{i,j=1}^n
\lambda_i \lambda_j
e^{\mathrm{i} \omega^\top s_i}
e^{ - \mathrm{i} \omega^\top   s_j}
\right)
d \omega
\notag
\\
& = 
\int_{\mathbb{R}^d}
\hat{k}_{\theta}(\omega)
\left|
\sum_{i=1}^n
\lambda_i 
e^{\mathrm{i} \omega^\top s_i}
\right|^2
d \omega,
\end{align}
where $|z|$ is the modulus of a complex number $z$. In \eqref{eq:quadratic:form:FT}, $\hat{k}_{\theta}(\omega)$ is strictly positive. Furthermore, because $s_1, \ldots, s_n$ are two-by-two distinct, the family of functions $(\omega \mapsto e^{\mathrm{i} \omega^\top s_i})_{i=1,\ldots,n}$ is linearly independent. Hence, $\sum_{i,j=1}^n
\lambda_i \lambda_j (R_{\theta})_{i,j} >0$ for $(\lambda_1,\ldots,\lambda_n) \neq 0$. This shows that $\lambda_{\inf}(R_{\theta}) >0$ for $n \in \mathbb{N}$ and $\theta \in \Theta$. Proving that the infimum in the lemma is also strictly positive is also based on \eqref{eq:quadratic:form:FT}. We refer to the proofs of  Proposition D.4 in \cite{bachoc14asymptotic} or of Theorem 5 in \cite{bachoc2016smallest}.
\end{proof}

The next lemma will enable to control the variance of the likelihood criterion and the order of magnitude of its derivatives.

\begin{lemma}  \label{lem:lik:and:grad:converge}
Assume that \eqref{eq:smallest:distance:Delta}, \eqref{eq:sup:k:theta}, \eqref{eq:sup:partial:k:theta} and \eqref{eq:condition:FT} hold.
For any $\theta \in \Theta$, as $n \to \infty$,
\[
\mathrm{var}(L_n(\theta)) = o(1).
\]
Furthermore
\[
\max_{i=1,\ldots,p}
\sup_{\theta \in \Theta}
\left|
 \frac{\partial}{\partial \theta_i} L_n(\theta)
 \right|
 = O_p(1).
\]
\end{lemma}
\begin{proof}[sketch]
Using that $y$ is a centered Gaussian vector, we have, with $\mathrm{cov}(z)$ the covariance matrix of a random vector $z$, from Appendix B in the supplementary material,
\[
\mathrm{var}(L_n(\theta))
=
\frac{1}{n^2}
\mathrm{var}(y^\top R_{\theta}^{-1} y)
= 
\frac{2}{n^2}
\mathrm{tr}
\left(
R_{\theta}^{-1} \mathrm{cov}(y)
R_{\theta}^{-1} \mathrm{cov}(y)
\right)
= 
\frac{2}{n^2}
\mathrm{tr}
\left(
R_{\theta}^{-1} R_{\theta_0}
R_{\theta}^{-1} R_{\theta_0}
\right).
\]
Let $\lambda_{\sup}(M)$ be the largest eigenvalue of a symmetric matrix $M$. From Gershgorin circle theorem, we have
\begin{align*}
\lambda_{\sup} (R_{\theta_0})
& \leq 
\max_{i=1,\ldots,n}
\sum_{j=1}^n \left| (R_{\theta_0})_{i,j} \right|
\\
& = 
\max_{i=1,\ldots,n}
\sum_{j=1}^n \left| k_{\theta_0}(s_i - s_j) \right|
\\
(\text{from} ~ \eqref{eq:sup:k:theta}:)
& \leq 
\max_{i=1,\ldots,n}
\sum_{j=1}^n 
\frac{C_{\sup}}{1+||s_i - s_j||^{d + C_{\inf}}}.
\end{align*}
It is shown in \cite{bachoc14asymptotic} that \eqref{eq:smallest:distance:Delta} implies that 
\[
\max_{i=1,\ldots,\infty}
\sum_{j=1}^{\infty}
\frac{C_{\sup}}{1+||s_i - s_j||^{d + C_{\inf}}}
< \infty.
\]
Hence there is a constant $A_1 < \infty$ such that $\lambda_{\sup}( R_{\theta_0} ) \leq A_1$. Also, from Lemma \ref{lem:smallest:eigenvalue}, there is a constant $A_2 < \infty$ such that $\sup_{\theta \in \Theta} \lambda_{\sup}( R_{\theta}^{-1} ) \leq A_2$. 
Hence, we have $\mathrm{var}(L_n(\theta)) \leq  2 A_1^2 A_2^2  /n$ which proves the first part of the lemma.

For the second part of the lemma, let $\rho_{\sup}(M)$ be the largest singular value of a matrix $M$. Using Gershgorin circle theorem again, together with \eqref{eq:sup:partial:k:theta}, we show that there is a constant $A_3 < \infty$ such that,
\[
\max_{i=1,\ldots,p}
\sup_{\theta \in \Theta} \rho_{\sup}
\left(
  \frac{
  \partial R_{\theta}
  }{
  \partial \theta_i 
  }
  \right)
  \leq A_3.
\]
With this, we have
\begin{align*}
\max_{i=1,\ldots,p}
\sup_{\theta \in \Theta}
\left|
 \frac{\partial}{\partial \theta_i} L_n(\theta)
 \right|
 &
=
\max_{i=1,\ldots,p}
\sup_{\theta \in \Theta}
\left|
\frac{1}{n}
\mathrm{tr}
\left(
R_{\theta}^{-1}
\frac{\partial R_{\theta}}{ \partial \theta_i}
\right)
-
\frac{1}{n}
y^\top 
 R_{\theta}^{-1}
 \frac{\partial R_{\theta}}{ \partial \theta_i}
  R_{\theta}^{-1}
  y
  \right|
  \\
  & \leq 
A_2   A_3
+
A_2^2 A_3 
\frac{||y||^2}{n}.
\end{align*}
This last quantity is a $O_p(1)$ because $||y||^2/n$ is non-negative with (bounded) expectation $\mathrm{var}(\xi(0))$.
\end{proof}

The consistency result will rely on the following asymptotic identifiability assumption. We assume that for all $\epsilon >0$,
\begin{equation} \label{eq:identifiability:global}
\liminf_{n \to \infty}
\inf_{\substack{\theta \in \Theta \\ ||\theta - \theta_0|| \geq \epsilon}}
\frac{1}{n}
\sum_{i,j=1}^n
\left(
k_{\theta}(s_i - s_j)
-
k_{\theta_0}(s_i - s_j)
\right)^2
> 0.
\end{equation}
This assumption means that for $\theta$ bounded away from $\theta_0$, there is sufficient information in the spatial locations $s_1,\ldots,s_n$ to distinguish between the two covariance functions $k_{\theta}$ and $k_{\theta_0}$. In \cite{bachoc14asymptotic}, an explicit example is provided for which \eqref{eq:identifiability:global} holds. 

We remark that, even though there are $n^2$ terms in the sum in \eqref{eq:identifiability:global}, this sum can be shown to be a $O(n)$ for any fixed $\theta \in \Theta$, because of \eqref{eq:sup:k:theta} (by proceeding as in the proof of Lemma \ref{lem:lik:and:grad:converge}). The intuition is that, asymptotically, for many pairs $i,j \in \{1, \ldots,n \}$, $k_{\theta}(s_i - s_j)$
and
$k_{\theta_0}(s_i - s_j)$
are small. This is why the normalization factor is $1/n$ rather than $1/n^2$ in \eqref{eq:identifiability:global}.

With the assumption \eqref{eq:identifiability:global}, we can now state the consistency result.

\begin{theorem}[\cite{bachoc14asymptotic}]
Assume that \eqref{eq:smallest:distance:Delta}, \eqref{eq:sup:k:theta}, \eqref{eq:sup:partial:k:theta}, \eqref{eq:condition:FT} and \eqref{eq:identifiability:global} hold.
As $n \to \infty$
\[
\hat{\theta}_{\text{ML}}
\to^p 
\theta_0.
\]
\end{theorem}

\begin{proof}[sketch]
From Lemma \ref{lem:lik:and:grad:converge} we have, for any $\theta \in \Theta$,
\[
L_n(\theta) - \mathbb{E}( L_n(\theta) ) \to^p_{n \to \infty} 0. 
\]
Furthermore one can show, similarly as in Lemma \ref{lem:lik:and:grad:converge},
\[
\max_{i=1,\ldots,p}
\sup_{\theta \in \Theta}
\left|
 \frac{\partial}{\partial \theta_i} \mathbb{E} ( L_n(\theta))
 \right|
 = O(1).
\]
Hence, using Lemma \ref{lem:lik:and:grad:converge}, we obtain
\begin{equation} \label{eq:lik:mius:exp:unif}
\sup_{\theta \in \Theta}
\left|
  L_n(\theta)
 -
  \mathbb{E} ( L_n(\theta))
 \right|
 = o_p(1).
\end{equation}
Next, it is shown in \cite{bachoc14asymptotic} that there exists a constant $A_4 > 0$ such that for $\theta \in \Theta$
\begin{equation} \label{eq:diff:exp:larger}
  \mathbb{E} ( L_n(\theta))
  -
    \mathbb{E} ( L_n(\theta_0))
    \geq 
    A_4
    \frac{1}{n}
\sum_{i,j=1}^n
\left(
k_{\theta}(s_i - s_j)
-
k_{\theta_0}(s_i - s_j)
\right)^2.
\end{equation}
From \eqref{eq:diff:exp:larger} and \eqref{eq:identifiability:global}, we then obtain, for $\epsilon >0$, with a strictly positive constant $A_5$, for $n$ large enough,
\begin{equation} \label{eq:for:consistency}
\inf_{\substack{\theta \in \Theta \\ ||\theta - \theta_0|| \geq \epsilon}}
\left(
 \mathbb{E} ( L_n(\theta))
  -
    \mathbb{E} ( L_n(\theta_0))
    \right)
    \geq 
    A_5.
\end{equation}
Combining \eqref{eq:lik:mius:exp:unif} and \eqref{eq:for:consistency} enables to conclude the proof with a standard M-estimator argument (for instance as in the proof of Theorem 5.7 in \cite{van2000asymptotic}). 
\end{proof}

\subsection{Asymptotic normality}

For $i \in \{ 1 , \ldots , p \}$, we have seen in the proof of Lemma \ref{lem:lik:and:grad:converge} that the $i$-th partial derivative of $L_n$ at $\theta_0$ is
\[
\frac{\partial}{\partial \theta_i} L_n(\theta_0)
=
\frac{1}{n}
\mathrm{tr}
\left(
R_{\theta_0}^{-1}
\frac{\partial R_{\theta_0}}{ \partial \theta_i}
\right)
-
\frac{1}{n}
y^\top 
 R_{\theta_0}^{-1}
 \frac{\partial R_{\theta_0}}{ \partial \theta_i}
  R_{\theta_0}^{-1}
  y.
\]
Since $y$ is a centered Gaussian vector and using Appendix B in the supplementary material, the element $i,j$ of the covariance matrix of the gradient of $L_n$ at $\theta$ is thus, for $i,j=1,\ldots,p$,
\begin{align} \label{eq:cov:score}
\mathrm{cov}
\left(
\frac{\partial}{\partial \theta_i} L_n(\theta_0)
,
\frac{\partial}{\partial \theta_j} L_n(\theta_0)
\right)
& =
\frac{2}{n^2}
\mathrm{tr}
\left(
 R_{\theta_0}^{-1}
 \frac{\partial R_{\theta_0}}{ \partial \theta_i}
  R_{\theta_0}^{-1}
    R_{\theta_0}
     R_{\theta_0}^{-1}
 \frac{\partial R_{\theta_0}}{ \partial \theta_j}
  R_{\theta_0}^{-1}
    R_{\theta_0}
\right) \notag \\
& =
\frac{2}{n^2}
\mathrm{tr}
\left(
 R_{\theta_0}^{-1}
 \frac{\partial R_{\theta_0}}{ \partial \theta_i}
     R_{\theta_0}^{-1}
 \frac{\partial R_{\theta_0}}{ \partial \theta_j}
\right).
\end{align}

It is shown in \cite{bachoc14asymptotic} that for $i,j\in \{1,\ldots,p\}$,
\[
\mathbb{E}
\left(
\frac{\partial^2}{\partial \theta_i \partial \theta_j} L_n(\theta_0)
\right)
=
\frac{1}{n}
\mathrm{tr}
\left(
 R_{\theta_0}^{-1}
 \frac{\partial R_{\theta_0}}{ \partial \theta_i}
     R_{\theta_0}^{-1}
 \frac{\partial R_{\theta_0}}{ \partial \theta_j}
\right).
\]

We will thus need to ensure that the $p \times p$ matrix with element $i,j$ equal to 
\[
\frac{1}{n}
\mathrm{tr}
\left(
 R_{\theta_0}^{-1}
 \frac{\partial R_{\theta_0}}{ \partial \theta_i}
     R_{\theta_0}^{-1}
 \frac{\partial R_{\theta_0}}{ \partial \theta_j}
\right)
\]
is asymptotically invertible. For this, we assume that for all  $(\lambda_1 , \ldots , \lambda_p) \in \mathbb{R}^p \backslash \{0\} $,
\begin{equation} \label{eq:identifiability:local}
\liminf_{n \to \infty}
\frac{1}{n}
\sum_{i,j=1}^n
\left(
\sum_{m=1}^p
\lambda_m
\frac{\partial k_{\theta_0}(s_i - s_j)}{\partial \theta_m}
\right)^2
> 0.
\end{equation}
This assumption is interpreted as a local identifiability condition around $\theta_0$. In \cite{bachoc14asymptotic}, an explicit example is provided for which \eqref{eq:identifiability:local} holds.

We can now state the asymptotic normality result for maximum likelihood estimators.

\begin{theorem}  \label{thm:increasing:normality}
Assume that \eqref{eq:smallest:distance:Delta}, \eqref{eq:sup:k:theta}, \eqref{eq:sup:partial:k:theta}, \eqref{eq:condition:FT}, \eqref{eq:identifiability:global} and \eqref{eq:identifiability:local} hold.
Let $\Sigma_{\theta_0}$ be the $p \times p$ matrix with element $i,j$ equal to
\[
\frac{1}{2}
\frac{1}{n}
\mathrm{tr}
\left(
 R_{\theta_0}^{-1}
 \frac{\partial R_{\theta_0}}{ \partial \theta_i}
     R_{\theta_0}^{-1}
 \frac{\partial R_{\theta_0}}{ \partial \theta_j}
\right).
\] 
Then 
\begin{equation} \label{eq:lambda:inf:lambda:sup}
0 < \liminf_{n \to \infty}\lambda_{\inf} (\Sigma_{\theta_0})
\leq
\limsup_{n \to \infty}
\lambda_{\sup} (\Sigma_{\theta_0})
< \infty.
\end{equation}
Furthermore, with $M^{-1/2}$ the unique symmetric matrix square root of $M^{-1}$ for a symmetric strictly positive definite $M$, we have
\begin{equation} \label{eq:the:TCL}
\sqrt{n}
\left( \Sigma_{\theta_0}^{-1}\right)^{-1/2} (\hat{\theta}_{\text{ML}} - \theta_0)
\to_{n \to \infty}^d
\mathcal{N}( 0 , I_p).
\end{equation}
\end{theorem}
We remark that in Theorem \ref{thm:increasing:normality}, $ \Sigma_{\theta_0}^{-1}$ is the asymptotic covariance matrix, but this matrix is not necessarily assumed to converge as $n \to \infty$. This matrix has its eigenvalues bounded away from zero and infinity asymptotically, so that the rate of convergence is $\sqrt{n}$ in Theorem \ref{thm:increasing:normality}.

\begin{remark}
Here the element $i,j$ of $n \Sigma_{\theta_0}$ is $n^2 / 4$ times the covariance between the elements $i$ and $j$ of the gradient of $L_n$, from \eqref{eq:cov:score}. Note that $L_n$ is $-2 / n$ times the log-likelihood (up to a constant not depending on $y$ or $\theta$). Consider now the score vector that is equal to the gradient of the log-likelihood. Then, we obtain that the covariance between the elements $i$ and $j$ of the score is $n^2 / 4$  times $4 / n^2$ times the element $i,j$ of $n \Sigma_{\theta_0}$.

In other words, $n \Sigma_{\theta_0}$ is the (theoretical) Fisher information matrix. In agreement with this, remark that from Theorem \ref{thm:increasing:normality} the inverse of $n \Sigma_{\theta_0}$ provides the asymptotic covariance matrix of maximum likelihood estimators as $n \to \infty$.
\end{remark}

\begin{proof}[sketch]
In \cite{bachoc14asymptotic}, it is shown that there exists a strictly positive constant $A_6$ such that for any $\lambda_1,\ldots,\lambda_p$ with $\lambda_1^2 + \dots + \lambda_p^2 =1$, we have
\[
\sum_{i,j=1}^p
\lambda_i \lambda_j
\frac{1}{2}
\frac{1}{n}
\mathrm{tr}
\left(
 R_{\theta_0}^{-1}
 \frac{\partial R_{\theta_0}}{ \partial \theta_i}
     R_{\theta_0}^{-1}
 \frac{\partial R_{\theta_0}}{ \partial \theta_j}
\right)
\geq 
A_6
\frac{1}{n}
\sum_{i,j=1}^n
\left(
\sum_{m=1}^p
\lambda_m
\frac{\partial k_{\theta_0}(s_i - s_j)}{\partial \theta_m}
\right)^2.
\]
Hence, from \eqref{eq:identifiability:local},
\[
0 < \liminf_{n \to \infty}\lambda_{\inf} (\Sigma_{\theta_0}).
\]
Hence  $\Sigma_{\theta_0}$ is invertible for $n$ large enough. Let $n$ be large enough so that this is the case in the rest of the proof.

One can show as in the proof of Lemma \ref{lem:lik:and:grad:converge} (see also \cite{bachoc14asymptotic})
that
\[
\limsup_{n \to \infty}
\lambda_{\sup} (\Sigma_{\theta_0})
< \infty.
\]

Hence \eqref{eq:lambda:inf:lambda:sup}
is proved. Let us now prove \eqref{eq:the:TCL}.

It is shown in \cite{bachoc14asymptotic} (see also \cite{bachoc2019asymptotic}), using a standard M-estimator argument together with techniques similar as above, that
\begin{align*}
\sqrt{n} ( \hat{\theta}_{\text{ML}} - \theta_0 )
&
=
-
\left(
\left[
\mathbb{E}
\left(
\frac{\partial^2}{\partial \theta_i \partial \theta_j} L_n(\theta_0)
\right)
\right]_{i,j=1,\ldots, p}
\right)^{-1}
\sqrt{n}
\left(
\frac{\partial}{\partial \theta_i } L_n(\theta_0)
\right)_{i=1,\ldots,p}
+o_p(1)
\\
& = 
-
\frac{1}{2}
\Sigma_{\theta_0}^{-1}
\sqrt{n}
\left(
\frac{\partial}{\partial \theta_i } L_n(\theta_0)
\right)_{i=1,\ldots,p}
+o_p(1).
\end{align*}
Hence to conclude the proof, it is sufficient to show that
\[
\left(
4
\Sigma_{\theta_0}
\right)^{-1/2}
\sqrt{n}
\left(
\frac{\partial}{\partial \theta_i } L_n(\theta_0)
\right)_{i=1,\ldots,p}
\to_{n \to \infty}^d
\mathcal{N}(0,I_p).
\]

Let us show this using linear combinations.
Let us write the $p \times 1$ gradient vector
\[
\frac{\partial}{\partial \theta } L_n(\theta_0)
=
\left(
\frac{\partial}{\partial \theta_i } L_n(\theta_0)
\right)_{i=1,\ldots,p}.
\]
Let $\lambda = (\lambda_1,\ldots,\lambda_p)^\top \in \mathbb{R}^p$ be fixed with $\lambda_1^2 + \dots + \lambda_p^2 = 1$. We have
\begin{align*}
\sum_{i=1}^p
\lambda_i 
\left(
\left(
4
\Sigma_{\theta_0}
\right)^{-1/2}
\sqrt{n}
\frac{\partial}{\partial \theta } L_n(\theta_0)
\right)_i
& = 
\sum_{i=1}^p
\left(
\left(
4
\Sigma_{\theta_0}
\right)^{-1/2}
\lambda
\right)_i
\sqrt{n}
\frac{\partial}{\partial \theta_i } L_n(\theta_0).
\\
\end{align*}
Let us now write $\beta_i = \left(
\left(
4
\Sigma_{\theta_0}
\right)^{-1/2}
\lambda
\right)_i$. We have
\begin{align*}
&
\sum_{i=1}^p
\lambda_i 
\left(
\left(
4
\Sigma_{\theta_0}
\right)^{-1/2}
\sqrt{n}
\frac{\partial}{\partial \theta } L_n(\theta_0)
\right)_i
\\
& =
\sum_{i=1}^p
\beta_i
\sqrt{n}
\frac{\partial}{\partial \theta_i } L_n(\theta_0)
\\
& = 
-
\sqrt{n} 
\left(
y^\top 
\left(
\frac{1}{n}
\sum_{i=1}^p
\beta_i
R_{\theta_0}^{-1}
\frac{\partial R_{\theta_0}}{\partial \theta_i}
R_{\theta_0}^{-1}
\right)
y
-
\mathbb{E}
\left(
y^\top 
\left(
\frac{1}{n}
\sum_{i=1}^p
\beta_i
R_{\theta_0}^{-1}
\frac{\partial R_{\theta_0}}{\partial \theta_i}
R_{\theta_0}^{-1}
\right)
y
\right)
\right),
\end{align*}
using for the last equality that the gradient of the logarithm of the likelihood at $\theta_0$ has mean zero.
Letting $z = (z_1,\ldots,z_n)^\top = R_{\theta_0}^{-1/2} y$, the negative of the above quantity is equal to
\begin{equation} \label{eq:for:TCL:one:d}
\sqrt{n}
\left( 
z^\top 
\left(
\frac{1}{n}
\sum_{i=1}^p
\beta_i
R_{\theta_0}^{-1/2}
\frac{\partial R_{\theta_0}}{\partial \theta_i}
R_{\theta_0}^{-1/2}
\right)
z
-
\mathbb{E}
\left(
z^\top 
\left(
\frac{1}{n}
\sum_{i=1}^p
\beta_i
R_{\theta_0}^{-1/2}
\frac{\partial R_{\theta_0}}{\partial \theta_i}
R_{\theta_0}^{-1/2}
\right)
z
\right)
\right)
.
\end{equation}
Letting $\rho_1,\ldots,\rho_n$ be the eigenvalues of $
(1/n)
\sum_{i=1}^p
\beta_i
R_{\theta_0}^{-1/2}
\frac{\partial R_{\theta_0}}{\partial \theta_i}
R_{\theta_0}^{-1/2}$
and letting $w = (w_1,\ldots,w_n) \sim \mathcal{N}( 0 ,I_n)$,
\eqref{eq:for:TCL:one:d} is equal, in distribution, to
\begin{equation} \label{eq:for:lindeberg}
\sqrt{n}
\sum_{i=1}^n
(w_i^2-1) \rho_i.
\end{equation}
Let us show that \eqref{eq:for:lindeberg} converges to a standard Gaussian distribution. We have
\begin{align*}
\mathrm{var} \left( \sqrt{n}
\sum_{i=1}^n
(w_i^2-1) \rho_i \right)  
&
=
2 n \sum_{i=1}^n \rho_i^2
\\
& =
\mathrm{var}
\left(
\sum_{i=1}^p
\lambda_i 
\left(
\left(
4
\Sigma_{\theta_0}
\right)^{-1/2}
\sqrt{n}
\frac{\partial}{\partial \theta } L_n(\theta_0)
\right)_i
\right) 
\\
& =
\lambda^\top
\mathrm{cov} 
\left(
\left(
4
\Sigma_{\theta_0}
\right)^{-1/2}
\sqrt{n}
\frac{\partial}{\partial \theta } L_n(\theta_0)
\right)
\lambda
\\
(\text{from}~\eqref{eq:cov:score}:)  &  =   \lambda^\top I_p \lambda
\\
& = 1.
\end{align*}
One can show as in the proof of Lemma \ref{lem:lik:and:grad:converge} (see also \cite{bachoc14asymptotic}) that $\max_{i=1}^n | \rho_i | = O(1/n)$. Hence, the classical Lindeberg-Feller central limit theorem enables to conclude that \eqref{eq:for:lindeberg} converges to a standard Gaussian distribution (see also \cite{IL97}). This concludes the proof.
\end{proof}

To conclude Section \ref{section:increasing}, the consistency and asymptotic normality results given here are quite generally applicable to families of stationary covariance functions and to Gaussian processes with zero mean functions. Some extensions to non-zero constant mean functions are discussed in \cite{bachoc2019asymptotic}. It would be interesting to provide extensions to non-stationary covariance functions or to unknown non-constant mean functions, with a parametric family of mean functions. It is possible that some of the proof techniques and intermediary results presented in Section \ref{section:increasing} and in \cite{bachoc14asymptotic} would be relevant for these extensions. Nevertheless, new arguments would also need to be developed, and appropriate assumptions, on the non-stationary covariance functions and non-constant mean functions, would need to be considered.

\section{Fixed-domain asymptotics} \label{sec:fixed:domain}

\subsection{What changes}

Under fixed-domain asymptotics, the spatial locations $s_1 , \ldots , s_n$ are restricted to a compact set $D \subset \mathbb{R}^d$. In this case, almost none of the proof techniques above for increasing-domain asymptotics can be applied. Indeed, they are based on the fact that for a given $i \in \{1 , \ldots , n\}$, $\xi (s_i)$ has a very small covariance with $\xi(s_j)$ for most $s_j$, $j=1,\ldots,n$. On the contrary, under fixed-domain asymptotics, for instance if $k_{\theta_0}$ is non-zero on $\mathbb{R}^d$, $\xi(s_i)$ has a non negligible covariance with all the $\xi(s_j)$, $j=1,\ldots,n$.

In particular, contrary to Lemma \ref{lem:smallest:eigenvalue}, if $\theta \in \Theta$ is such that $k_{\theta}$ is continuous at zero, then the smallest eigenvalue of $R_{\theta}$ goes to zero as $n \to \infty$.
This is seen by considering a sequence of $2 \times 2$ submatrices based on $s_{i_n} , s_{j_n}$ with $||s_{i_n} - s_{j_n}|| \to 0$ as $n \to \infty$. Similarly, the largest eigenvalue of $R_{\theta}$ goes to infinity as $n \to \infty$ for any $\theta \in \Theta$ if $k_{\theta}$ is, for instance, non-zero on $\mathbb{R}^d$. 

\subsection{Microergodic and non-microergodic parameters}

The conclusion of Section \ref{section:increasing} on increasing-domain asymptotics is that the family of stationary covariance functions $\{k_{\theta} ; \theta \in \Theta\}$ can be fairly general to prove the consistency and asymptotic normality of maximum likelihood estimators of $\theta_0$. In particular, under the reasonable conditions \eqref{eq:identifiability:global} and \eqref{eq:identifiability:local}, $\theta_0$ can be entirely consistently estimable.

We will now see that, in contrast, for a family $\{k_{\theta} ; \theta \in \Theta\}$ of covariance functions, under fixed-domain asymptotics, it can regularly be the case that $\theta_0$ is not entirely consistently estimable.

The notion that makes this more precise is that of the equivalence of Gaussian measures \cite{ibragimov78gaussian,stein99interpolation}. Consider two covariance parameters $\theta_1 , \theta_2 \in \Theta$, $\theta_1 \neq \theta_2$. If $\xi$ has covariance function $k_{\theta_1}$, $\xi$ yields a measure $\mathcal{M}_{\theta_1}$ on the set of functions from $D$ to $\mathbb{R}$, with respect to the cylindrical sigma-algebra\footnote{If Gaussian processes with continuous realizations on compact sets are considered, one can also define Gaussian measures over the Banach space of continuous functions (on a compact set) endowed with the supremum norm and the corresponding Borel sigma-algebra.}. Similarly, if $\xi$ has covariance function $k_{\theta_2}$, $\xi$ yields a measure $\mathcal{M}_{\theta_2}$. When $D$ is compact, these two measures can be equivalent (for a set $A$ of functions, $\mathcal{M}_{\theta_1}(A) = 0$ if and only if $\mathcal{M}_{\theta_2}(A) = 0$) even when the covariance functions $k_{\theta_1}$ and $k_{\theta_2}$ are different.

The notion of equivalence of Gaussian measures enables to define non-microergodic parameters.

\begin{definition} \label{def:non-micro}
Let $\Phi$ be a function from $\Theta$ to $\mathbb{R}^q$ for $q \in \mathbb{N}$. We say that $\Phi(\theta_0)$ is non-microergodic if there exists $\theta_1 \in \Theta$ such that $\Phi(\theta_1) \neq \Phi(\theta_0)$ and the measures $\mathcal{M}_{\theta_1}$ and $\mathcal{M}_{\theta_0}$ are equivalent.
\end{definition}

If a covariance parameter is non-microergodic, it can not be estimated consistently.

\begin{lemma}
Let $(s_i)_{i \in \mathbb{N}}$ be any sequence of points in $D$.
If $\Phi(\theta_0)$ is non-microergodic, there does not exist a sequence of functions $\hat{\Phi}_n : \mathbb{R}^n \to \mathbb{R}^q$ such that, for any $\theta \in \Theta$, if $\xi$ has covariance function $k_{\theta}$ then $ \hat{\Phi}_n (\xi(s_1) , \ldots , \xi(s_n))$ goes to $\Phi(\theta)$ in probability as $n \to \infty$.
\end{lemma}

\begin{proof}
Let $\Phi(\theta_0)$ be non-microergodic. Then fix $ \theta_1 \in \Theta$ such that $\Phi(\theta_1) \neq \Phi(\theta_0)$ and the measures $\mathcal{M}_{\theta_1}$ and $\mathcal{M}_{\theta_0}$ are equivalent.

Assume that  an estimator sequence $\hat{\Phi}_n$ as described in the lemma exists. Then, when $\xi$ has covariance function $k_{\theta_0}$, as $n \to \infty$,
\[
\hat{\Phi}_n (\xi(s_1) , \ldots , \xi(s_n)) \to^p \Phi(\theta_0).
\]
Hence there exists a subsequence $n'$ such that as $n' \to \infty$, almost surely,
\[
\hat{\Phi}_{n'} (\xi(s_1) , \ldots , \xi(s_{n'})) \to \Phi(\theta_0).
\]
This can be written in the form
\[
\mathcal{M}_{\theta_0}
\left(
\left\{
f ~ \text{function from $D$ to $\mathbb{R}$ such that}~
\hat{\Phi}_{n'} (f(s_1) , \ldots , f(s_{n'})) \to_{n' \to \infty} \Phi(\theta_0)
\right\}
\right)
=1.
\]
Then since the measures $\mathcal{M}_{\theta_1}$ and $\mathcal{M}_{\theta_0}$ are equivalent
\[
\mathcal{M}_{\theta_1}
\left(
\left\{
f ~ \text{function from $D$ to $\mathbb{R}$ such that}~
\hat{\Phi}_{n'} (f(s_1) , \ldots , f(s_{n'})) \to_{n' \to \infty} \Phi(\theta_0)
\right\}
\right)
=1.
\]
This means that, when $\xi$ has covariance function $k_{\theta_1}$, the sequence $\hat{\Phi}_{n'} (\xi(s_1) , \ldots , \xi(s_{n'}))$ goes almost surely to $\Phi(\theta_0) \neq \Phi(\theta_1)$. Hence the sequence 
$\hat{\Phi}_{n} (\xi(s_1) , \ldots , \xi(s_{n}))$ does not go to $\Phi(\theta_1)$ in probability as $n \to \infty$. This is a contradiction which concludes the proof.
\end{proof}

Hence, one should not expect to have accurate estimators of non-microergodic parameters under fixed-domain asymptotics. The interpretation of non-microergodic parameters is that, even if $\Phi(\theta_0)$ and $\Phi(\theta_1)$ are different, there is not enough information  in a single realization of the random function $\{\xi(s) ; s \in D\}$ (even if this realization was observed continuously) to distinguish between $\Phi(\theta_0)$ and $\Phi(\theta_1)$. This lack of information stems from the boundedness of $D$.

It is important to remark that there exist results showing that non-microergodic parameters have an asymptotically negligible impact on prediction of unknown values of $\xi$ \cite{AEPRFMCF,BELPUICF,UAOLPRFUISOS,zhang04inconsistent}.
In \cite{stein99interpolation}, this situation is interpreted as an instance of the following principle, called Jeffreys's law: ``things we shall never find much out about cannot be very important for prediction''.

Finally, we can define microergodic parameters.
\begin{definition}
Let $\Phi$ be a function from $\Theta$ to $\mathbb{R}^q$ for $q \in \mathbb{N}$. We say that $\Phi(\theta_0)$ is microergodic if for any $\theta_1 \in \Theta$ such that $\Phi(\theta_1) \neq \Phi(\theta_0)$, the measures $\mathcal{M}_{\theta_1}$ and $\mathcal{M}_{\theta_0}$ are orthogonal (i.e. there exists a set of functions $A$ such that $\mathcal{M}_{\theta_1}(A) =0$ and $\mathcal{M}_{\theta_0} (A)$ = 1).
\end{definition}

\subsection{Consistent estimation of the microergodic parameter of the isotropic Mat\'ern model} \label{subsection:consistency:matern}

Let us now focus on the family of isotropic Mat\'ern covariance functions \eqref{eq:isotropic:matern}, in the case where the smoothness parameter $\nu$ is known. We thus consider $\theta = (\sigma^2 , \alpha) \in \Theta = (0 , \infty) \times [\alpha_{\inf} , \alpha_{\sup}]$ with $0 < \alpha_{\inf} < \alpha_{\sup} < \infty$ fixed. We thus have
\begin{equation} \label{eq:matern:nu:known}
k_{\theta} (x) = \frac{\sigma^2 2^{1-\nu}}{\Gamma(\nu)} \left (\alpha||x||
  \right )^{\nu} {\cal K}_{\nu} \left (\alpha ||x|| \right ),
\end{equation}
for $x \in \mathbb{R}^d$ where $0 <\nu < \infty$ is fixed and known. We let $\theta_0  = (\sigma_0^2 , \alpha_0)$. In the rest of Section \ref{sec:fixed:domain}, we set the dimension as $d \in \{1,2,3\}$.

Then the parameters $\sigma_0^2$ and $\alpha_0$ are non-microergodic, while the parameter $\sigma_0^2 \alpha_0^{2 \nu}$ is microergodic.

\begin{proposition}[\cite{zhang04inconsistent}] \label{prop:microergodic:matern}
With the family of covariance functions given by \eqref{eq:matern:nu:known}, the measures $\mathcal{M}_{\theta_1}$ and  $\mathcal{M}_{\theta_0}$ are equivalent if $\sigma_1^2 \alpha_1^{2 \nu} = \sigma_0^2 \alpha_0^{2 \nu}$ and are orthogonal if $\sigma_1^2 \alpha_1^{2 \nu} \neq \sigma_0^2 \alpha_0^{2 \nu}$. Hence, $\sigma_0^2 \alpha_0^{2 \nu}$ is microergodic, and in particular $\sigma_0^2$ and $\alpha_0$ are non-microergodic.
\end{proposition}

We remark that Proposition \ref{prop:microergodic:matern} holds for $d \in \{ 1 , 2 ,3\}$, which is the ambient assumption in Section \ref{subsection:consistency:matern}. When $d \geq 5$, \cite{anderes2010consistent} proved that the full parameter $(\sigma_0^2 , \alpha_0)$ is microergodic (thus in particular $\sigma_0^2$ and  $\alpha_0$ are microergodic). At the time of \cite{anderes2010consistent}, it was mentioned there that the case $d=4$ was open, that is, it was not known if $\sigma_0^2$ and  $\alpha_0$ are microergodic in this case. Currently, this case is still open, to the best of our knowledge.

Then, \cite{zhang04inconsistent} finds a consistent estimator of $\sigma_0^2 \alpha_0^{2 \nu}$ by fixing $\alpha$ to an arbitrary value and by maximizing the likelihood with respect to $\sigma^2$ only. Hence, for $\alpha \in [\alpha_{\inf} , \alpha_{\sup}]$, let
\[
\hat{\sigma}^2(\alpha)
= \underset{\sigma^2 \in (0,\infty)}{\mathrm{argmin}}
L_n(\sigma^2 , \alpha).
\]
We remark that the $\mathrm{argmin}$ is unique from \eqref{eq:expression:lik:variance} in the proof of Theorem \ref{theorem:fixed:domain:consistency}.
By canceling the derivative of $L_n(\sigma^2 , \alpha)$ with respect to $\sigma^2$, we find
\begin{equation} \label{eq:hat:sigma:matern}
\hat{\sigma}^2 (\alpha)
=
\frac{1}{n}
y^\top 
\Sigma_{\alpha}^{-1}
y,
\end{equation}
with $\Sigma_{\alpha} = R_{\sigma^2 , \alpha} / \sigma^2$, based on \eqref{eq:expression:lik:variance} in the proof of Theorem \ref{theorem:fixed:domain:consistency}.

\begin{theorem}[\cite{zhang04inconsistent}] \label{theorem:fixed:domain:consistency}
Let $\alpha_1$ be any fixed element of $[\alpha_{\inf} , \alpha_{\sup}]$. 
As $n \to \infty$, almost surely,
\[
\hat{\sigma}^2(\alpha_1) \alpha_1^{2 \nu}
\to 
\sigma_0^2 \alpha_0^{2 \nu}.
\]
\end{theorem}
\begin{proof}[sketch]
 Let
\[
\sigma_1^2 = \frac{\sigma_0^2 \alpha_0^{2 \nu}}{\alpha_1^{2 \nu}}.
\]
Let $\epsilon >0$.
 From Proposition \ref{prop:microergodic:matern}, the measures $\mathcal{M}_{\sigma_0^2,\alpha_0}$ and
 $\mathcal{M}_{\sigma_1^2,\alpha_1}$ are equivalent and the measures $\mathcal{M}_{\sigma_0^2,\alpha_0}$ and
 $\mathcal{M}_{\sigma_1^2 + \epsilon,\alpha_1}$ are orthogonal. Hence, \cite{zhang04inconsistent}, based on \cite{gikhman2004theory}, obtains that, almost surely,
 \[
n  L_n(\sigma_1^2 + \epsilon,\alpha_1)
 -
n  L_n( \sigma_1^2  , \alpha_1 )
 \to 
 \infty.
 \]
 Similarly, we can show that, almost surely,
 \[
 n L_n(\sigma_1^2 - \epsilon,\alpha_1)
 -
 n L_n( \sigma_1^2  , \alpha_1 )
 \to 
 \infty.
 \]
Let $\Sigma_{\alpha_1} = R_{\sigma^2 , \alpha_1} / \sigma^2$. 
Then
\begin{equation} \label{eq:expression:lik:variance}
L_n(\sigma^2 , \alpha_1)
=
\log(\sigma^2)
+
\frac{1}{n}
\log(| \Sigma_{\alpha_1} |)
+
\frac{1}{\sigma^2}
\frac{1}{n}
y^\top
\Sigma_{\alpha_1}^{-1}
y.
\end{equation}
Hence, $1/\sigma^2 \mapsto n L_n(\sigma^2 , \alpha_1)$ is convex, and thus by convexity we obtain, as $n \to \infty$,
\[
\left(
\inf_{\substack{ \sigma^2 \in (0 , \infty) \\ |\sigma^2 - \sigma_1^2 | \geq \epsilon}}
n L_n(\sigma^2 , \alpha_1)
\right)
- 
n L_n(\sigma_1^2 , \alpha_1)
\to \infty
\] 
almost surely. This implies that $\hat{\sigma}^2(\alpha_1) \to \sigma_1^2$ almost surely as $n \to \infty$ which concludes the proof.
\end{proof}

In the supplementary material, still for the Mat\'ern covariance functions, we also provide asymptotic normality results for the estimator $\hat{\sigma}^2(\alpha_1) \alpha_1^{2 \nu}$ and for the ``full'' maximum likelihood estimator, where the likelihood is maximized with respect to both $\sigma^2$ and $\alpha$.

We remark that, in general and outside of the Mat\'ern case, consistency results for maximum likelihood under fixed-domain asymptotics are quite scarce. We mention a few such other consistency results at the end of Appendix A in the supplementary material and in Section \ref{sec:conclusion}.

\section{Conclusion} \label{sec:conclusion}

We have presented some asymptotic results on covariance parameter estimation under increasing and fixed-domain asymptotics. The presentation highlights the strong differences between the two settings. Under increasing-domain  asymptotics, with mild identifiability conditions, all the components of the covariance parameter can be estimated consistently, and with asymptotic normality. The proof techniques hold for general families of stationary covariance functions. They are based on the asymptotic independence between most pairs of observations, as $n \to \infty$, that enables to control the logarithm of the likelihood and its gradient and to apply general methods for M-estimators.

In contrast, under fixed-domain asymptotics, typically all pairs of observations have a covariance that is not small.
As a consequence some components of  the covariance parameter can not be estimated consistently, even if changing the component changes the covariance function. The notion of equivalence of Gaussian measures, yielding the notion of microergodicity, is central. The results and proofs are not general in the current state of the literature. Here we have presented results and proofs related to the family of isotropic Mat\'ern covariance functions in dimension $d=1,2,3$. The presented proofs rely on the Fourier transforms of these covariance functions (through the results taken from the cited references) and also on the explicit expression of the logarithm of the likelihood as a function of the variance parameter $\sigma^2$.

There are many other existing contributions in the literature that we have not presented here. Under increasing-domain asymptotics, earlier results on maximum likelihood were provided by \cite{mardia84maximum}, using general results from \cite{UANMLE} (the latter not necessarily considering Gaussian processes). Restricted maximum likelihood was then studied in \cite{cressie93asymptotic}. Cross validation was considered in \cite{bachoc14asymptotic,bachoc2018asymptotic}. Extensions to transformed Gaussian processes were studied in \cite{bachoc2019asymptotic}. Pairwise likelihood was studied in \cite{bevilacqua15comparing}. Multivariate processes were considered in
\cite{furrer2016asymptotic,shaby12tapered}. Finally, more generally, the increasing-domain asymptotic framework is investigated in spatial statistics for instance in \cite{hallin09local,lahiri16central,Lahiri03central,Lahiri04asymptotic}.

Under fixed-domain asymptotics, earlier results for the estimation of the microergodic parameter in the family of exponential covariance functions in dimension one were obtained in \cite{Ying91}. The estimation of parameters for the Brownian motion is addressed in \cite{CGCVMMLEPSP}.
Various additional results on maximum likelihood are obtained in \cite{FDASMTGRF,ESCMSGRFM,VanDerVaart96maximum,ying93maximum}.
Variation-based estimators are studied in \cite{anderes2010consistent,blanke2014global,IL97,loh2015estimating}.
Composite likelihood is addressed in \cite{bachoc2019composite}.
The case of covariance parameter estimation for 
constrained Gaussian processes is addressed in \cite{bachoc2019maximum,lopez2018finite}. Cross validation is addressed in \cite{bachoc2017cross}.
Finally, extensions of  the fixed-domain asymptotic results presented here to the family of isotropic Wendland covariance functions are provided in \cite{bevilacqua2019estimation}.

\appendix

\section{Asymptotic normality for the estimation of the microergodic parameter of the isotropic Mat\'ern model}

In the case of the Mat\'ern covariance functions, in \cite{DuZhaMan2009}, a central limit theorem is proved for the same estimator $\hat{\sigma}^2(\alpha_1) \alpha_1^{2 \nu}$ as in \cite{zhang04inconsistent}. As in Section 4.3, we let $d \in \{1,2,3\}$.

\begin{theorem}[\cite{DuZhaMan2009}] \label{theorem:fixed:domain:normality}
Let $\alpha_1$ be any fixed element of $[\alpha_{\inf} , \alpha_{\sup}]$. 
As $n \to \infty$, 
\[
\sqrt{n} \left(
\hat{\sigma}^2(\alpha_1) \alpha_1^{2 \nu}
- 
\sigma_0^2 \alpha_0^{2 \nu}
\right)
\to_{n \to \infty}^d
\mathcal{N}( 0 , 2 (\sigma_0^2 \alpha_0^{2 \nu})^2).
\]
\end{theorem}
\begin{proof}[{\bf Proof (sketch)}]
In \cite{DuZhaMan2009}, it is shown (after an involved and technical proof that is based in particular on the Fourier transform expression of the Mat\'ern covariance function) that
\[
\hat{\sigma}^2(\alpha_1)
\alpha_1^{2 \nu}
-
\hat{\sigma}^2(\alpha_0)
\alpha_0^{2 \nu}
=
o_p
\left(
\frac{1}{\sqrt{n}}
\right).
\] 
Hence,
\begin{equation} \label{eq:for:using:slutsky}
\sqrt{n} \left(
\hat{\sigma}^2(\alpha_1) \alpha_1^{2 \nu}
- 
\sigma_0^2 \alpha_0^{2 \nu}
\right)
=
\sqrt{n} \left(
\hat{\sigma}^2(\alpha_0) \alpha_0^{2 \nu}
- 
\sigma_0^2 \alpha_0^{2 \nu}
\right)
+o_p(1).
\end{equation}
We have
\[
\sqrt{n} \left(
\hat{\sigma}^2(\alpha_0) \alpha_0^{2 \nu}
- 
\sigma_0^2 \alpha_0^{2 \nu}
\right)
=
\sigma_0^2 \alpha_0^{2 \nu}
\frac{1}{\sqrt{n}}
\left(
y^\top R_{\theta_0}^{-1} y
-
\mathbb{E}
(
y^\top R_{\theta_0}^{-1} y
)
\right).
\]
Since $y^\top R_{\theta_0}^{-1} y$ is a sum of squares of independent standard Gaussian variables, we have
\[
\sqrt{n} \left(
\hat{\sigma}^2(\alpha_0) \alpha_0^{2 \nu}
- 
\sigma_0^2 \alpha_0^{2 \nu}
\right)
\to_{n \to \infty}^d
\mathcal{N}( 0 , 2 (\sigma_0^2 \alpha_0^{2 \nu})^2 ).
\]
Hence, from \eqref{eq:for:using:slutsky} and Slutsky's lemma,
\[
\sqrt{n} \left(
\hat{\sigma}^2(\alpha_1) \alpha_1^{2 \nu}
- 
\sigma_0^2 \alpha_0^{2 \nu}
\right)
\to_{n \to \infty}^d
\mathcal{N}( 0 , 2 (\sigma_0^2 \alpha_0^{2 \nu})^2 ).
\]
This concludes the proof.
\end{proof}

Finally, the above central limit theorem relies on an arbitrary fixed choice of $\alpha_1$. Later, this central limit theorem was refined by  \cite{ShaKau2013}, based on intermediary results from \cite{WanLoh2011}, to allow for an arbitrary estimator of $\alpha_0$. More precisely, \cite{ShaKau2013} proves the following.

\begin{theorem}[\cite{ShaKau2013}] \label{theorem:fixed:domain:normality:two}
Let $(\hat{\alpha}_n)_{n \in \mathbb{N}}$ be any sequence of random variables in $[\alpha_{\inf} , \alpha_{\sup}]$. Then 
as $n \to \infty$, 
\[
\sqrt{n} \left(
\hat{\sigma}^2(\hat{\alpha}_n) \hat{\alpha}_n^{2 \nu}
- 
\sigma_0^2 \alpha_0^{2 \nu}
\right)
\to_{n \to \infty}^d
\mathcal{N}( 0 , 2 (\sigma_0^2 \alpha_0^{2\nu})^2).
\]
\end{theorem}

In particular, with a maximum likelihood estimator $(\hat{\sigma}^2_{\text{ML}},\hat{\alpha}_{\text{ML}})$, we have
\[
\hat{\sigma}_{\text{ML}}^2
=
\hat{\sigma}^2
(\hat{\alpha}_{\text{ML}})
\]
and thus Theorem \ref{theorem:fixed:domain:normality:two} implies 
\[
\sqrt{n} \left(
\hat{\sigma}_{\text{ML}}^2 \hat{\alpha}_{\text{ML}}^{2 \nu}
- 
\sigma_0^2 \alpha_0^{2 \nu}
\right)
\to_{n \to \infty}^d
\mathcal{N}( 0 , 2 (\sigma_0^2 \alpha_0^{2 \nu})^2),
\]
which is the asymptotic normality of maximum likelihood estimators of the microergodic parameter in the family of isotropic Mat\'ern covariance functions. Note that $\hat{\sigma}_{\text{ML}}^2$ and $\hat{\alpha}_{\text{ML}}$ typically do not converge separately to fixed quantities \cite{zhang05toward}.

We remark that the bounds $\alpha_{\inf}$ and $\alpha_{\sup}$ that define maximum likelihood estimators have no impact on Theorem \ref{theorem:fixed:domain:normality:two} as long as they are finite, non-zero and fixed independently of $n$. The proof techniques of \cite{ShaKau2013} do not allow for $\alpha_{\inf}$ and $\alpha_{\sup}$ depending on $n$ and going to zero or infinity. In practice, one also usually takes bounds $0 < \alpha_{\inf} <\alpha_{\sup} < \infty$ for $\alpha$ to implement maximum likelihood. Then a common practice is to take $\alpha_{\inf}$ and $\alpha_{\sup}$ of the same orders   as the inverses of the maximum and minimum distances between two distinct observation points.

Notice that all the results reviewed here for the Mat\'ern model assume $\nu$ to be fixed and known. We are not aware of any consistency results of maximum likelihood estimators of $\nu$ under fixed-domain asymptotics. Nevertheless, there exist other estimation techniques than maximum likelihood, that are shown to be able to estimate $\nu$ consistently, in particular variation-based estimators \cite{blanke2014global,IL97,loh2015estimating}.

We also remark that the results and proof techniques of Section 4 and Appendix A are intrinsically specific to the Mat\'ern covariance functions. Nevertheless, \cite{bevilacqua2019estimation} recently managed to extend them to the family of Wendland covariance functions. Finding other families of covariance functions for which similar extensions would be possible is an interesting topic for future research.

\section{Expectations and covariances of quadratic forms of a Gaussian vector}

Let $r \in \mathbb{N}$, let $V$ be a centered $r \times 1$ Gaussian vector and let $A$ and $B$ be fixed $r \times r$ matrices. Then we have
\begin{equation*} 
\mathbb{E}( V^\top A V ) = 
\mathrm{tr} \left(  A \mathrm{cov}(V) \right)
\end{equation*}
and
\begin{equation*} 
\mathrm{cov}( V^\top A V , V^\top B V ) = 
2 \mathrm{tr} \left(  A \mathrm{cov}(V)  B \mathrm{cov}(V) \right)
\end{equation*}
from, for instance, (A.6) and (A.7) in \cite{paolella2018linear}.

\end{document}